\documentclass[12pt, reqno]{amsart}
\usepackage{amsmath, amsthm, amscd, amsfonts, amssymb, graphicx, color}
\usepackage[bookmarksnumbered, colorlinks, plainpages]{hyperref}
\hypersetup{colorlinks=true,linkcolor=red, anchorcolor=green, citecolor=cyan, urlcolor=red, filecolor=magenta, pdftoolbar=true}

\usepackage{tensor}

\newtheorem{theorem}{Theorem}[section]
\newtheorem{lemma}[theorem]{Lemma}
\newtheorem{proposition}[theorem]{Proposition}
\newtheorem{corollary}[theorem]{Corollary}
\theoremstyle{definition}
\newtheorem{definition}[theorem]{Definition}
\newtheorem{example}[theorem]{Example}

\theoremstyle{remark}
\newtheorem{remark}[theorem]{Remark}
\numberwithin{equation}{section}

\usepackage{fullpage}

\begin{document}

\setcounter{page}{1}

\title[]{ $\eta$-metric structures}

\author{Ya\'e Ulrich Gaba}

\thanks{Electronic address: \texttt{gabayae2@gmail.com}}

\address{$^{1}$Institut de Math\'ematiques et de Sciences Physiques (IMSP)/UAC,
	Porto-Novo, B\'enin.}



\begin{abstract}
	In this paper, we discuss recent results about generalized metric spaces and fixed point theory. We introduce the notion of $\eta$-cone metric spaces, give some topological properties and prove some fixed point theorems for contractive type maps on these spaces. In particular we show that theses $\eta$-cone metric spaces are natural generalizations of both cone metric spaces and metric type spaces.
	\end{abstract}

\maketitle

\section{Introduction.}

Cone metric spaces were introduced in \cite{huang} and many fixed point results concerning
mappings in such spaces have been established. In \cite{khamsi}, M. A. Khamsi
connected this concept with a generalised form of metric space that he named \textit{metric type space} (MTS for short). Topological properties of metric type spaces and fixed point theorems for contractive mappings in metric type spaces  can extensively be read in \cite{cos,Gaba1,Gaba2,khamsi2}. The originality of the work by Huang\cite{huang} lies in the fact that they replace the real numbers by an ordered Banach space where the order on the underlying Banach space is defined via an associated cone subset. This work suggested different orientations in the generalization of the classical "metric spaces" (see for instance \cite{alex,li}).
The present manuscript investigates a similar extension by making a local assumption that will be clarified in the next lines. In fact, we replace the constant appearing in the classical "triangle inequality" of a MTS by a function of two variables. Then we study the topological properties of these new spaces and give some fixed point results. We conclude by showing that the new introduced $\eta$-cone metric spaces have a metric like structure.

\section{Basic definitions and preliminary results.}

First let us start by making some basic definitions.

\begin{definition}
	Let $E$ be a real Banach space with norm $\|.\|$ and $P$ be a subset of $E$. Then $P$ is called a cone if and only if 
	\begin{enumerate}
		\item $P$ is closed, nonempty and $P \neq \{\theta\}$, where $\theta $ is the zero vector in $E$;
		\item for any $a, b \geq 0$ (nonnegative real numbers), and $x, y \in P$, we have
		
		$ax + by \in P$;
		\item for $x \in P$, if $-x \in P$, then $x =  \theta$.
	\end{enumerate}
\end{definition}
Given a cone $P$ in a Banach space $E$, we define on $E$ a partial order $\preceq$ with respect to $P$ by
$$ x \preceq y \Longleftrightarrow y-x \in P. $$
We also write $x \prec y$ whenever $ x \preceq y$ and $x\neq y$, while $x\ll y$ will stand for $y-x \in Int(P)$ (where $Int(P)$ designates the interior of $P$).

The cone $P$ is said to be \textbf{normal} if there is a real number $C>0$, such that for all $x,y\in E$, we have
\[
\theta \preceq x \preceq y \Longrightarrow \|x\| \leq C \|y\|.
\]
The least positive number satisfying this inequality is called the \textbf{normal constant} of $P$. Therefore, we shall say that $P$ is a $K$-normal cone to indicate the fact that the normal constant is $K$.

The cone $P$ is said to be \textbf{regular} if every increasing sequence\footnote{Equivalently the
	cone $P$ is regular if and only if every decreasing sequence which is bounded from below
	is convergent.}
 which is bounded from above is convergent, i.e. if $(x_n)$ is a sequence such that $x_1\preceq x_2\preceq\cdots\preceq x_n\preceq \cdots \preceq y$ for some $y\in E$ then, there exists $x^* \in E$ such that $\underset{n\to \infty}{\lim} \| x_n-x^*\|=0$.

The cone $P$ is said to be \textbf{minihedral} cone if $\sup\{x, y\}$ exists for all $x, y \in E$,
and \textbf{strongly minihedral} if every subset of $E$ which is bounded from above has a
supremum and hence any subset of $E$ which is bounded from below has an infimum \cite{dei}. 

\vspace*{0.2cm}

Throughout this article we assume that the cone $P$ is normal with constant $K$ and $P$ is such that $int(P) \neq \emptyset$ and $\preceq$ is a partial ordering with respect to $P$. Hence the Banach space $E$ and the cone $P$ will be omitted and the Banach space $E$ will be assumed to be ordered with the order induced by the cone $P$.

\vspace*{0.2cm}

Now, we introduce a new type of generalized metric space, which we call $\eta$-cone metric spaces.

\begin{definition}
	Let $X$ be a non empty set and $\eta : X \times X \to [1, \infty)$ be a map. A function $d_\eta:X \times X \to E$ is called an \textbf{$\eta$-cone metric} on $X$ if:
	\begin{enumerate}
		\item[(d1)] $\theta \preceq d(x,y) \ \ \forall \ x \in X$ and $\ d_\eta(x,y)=\theta$ if and only if $x=y$; 
		\item[(d2)]$d_\eta(x,y)=d_\eta(y,x) \quad \forall\  x,y \in X$;
		\item[(d3)] $d_\eta(x,z) \preceq \eta(x,z)[d_\eta(x,y) + d_\eta(y,z)] \quad \forall\  x,y,z \in X $.
	\end{enumerate}
	The pair $(X,d_\eta)$ is called an \textbf{$\eta$-cone metric space}.
\end{definition}

\begin{remark} 
	If $\eta(x, y) = 1$ whenever $x,y\in X$, then we obtain the definition of a cone metric space (see \cite[Definition 1
]{huang}). For $\eta(x, y) = L(L\geq 1)$ whenever $x,y\in X$, then we obtain the definition of a cone metric type space (see \cite[Definition 2.1 ]{alex}).

In particular, when $E = \mathbb{R}$, $P = [0,+\infty)$ and $\eta(x,y)=C (C\geq 1)$ for all $x,y\in X$, then an $\eta$-cone metric space  reduces to a metric type space (MTS) (see \cite[Definiton 2.7]{khamsi}).
	
\end{remark}

\begin{example}\label{example1}
	Let $E = \mathbb{R}^2 , P = \{(x, y) \in E | x,y\geq  0\} \subseteq \mathbb{R}^2$ and $X=\{1,2,3\}.$ Let $\alpha \ge 0$ a constant and define $\eta : X \times X\to [0,\infty)$ and $d_\eta: X \times X \to E$
	as
	\[\eta(x,y)=1+x+y;\]
	\[d_\eta(1,1)=d_\eta(2,2)=d_\eta(3,3)=(0,0);\]
	
	\[ d_\eta(1,2)=d_\eta(2,1)=80(1,\alpha); d_\eta(1,3)=d_\eta(3,1)=1000(1,\alpha); d_\eta(2,3)=d_\eta(3,2)=600(1,\alpha).\]
	Then $(X, d_\eta )$ is an $\eta$-cone metric space.	
\end{example}

\begin{proof}
	(d1) and (d2) trivially hold. For (d3) we have:	
	
	\[d_\eta(1, 2) = 80(1,\alpha)\preceq \eta(1, 2)[d_\eta (1, 3) + d_\eta (3, 2)] = 4(1000 + 600)(1,\alpha) = 6400(1,\alpha),\]
	
	\[ d_\eta(1, 3) = 1000(1,\alpha)\preceq \eta(1, 3)[d_\eta (1, 2) + d_\eta(2, 3)] = 5(80 + 600)(1,\alpha) = 3400(1,\alpha).\]
	
	Similar calculations hold for $d_\eta(2, 3)$. Hence for all $x, y, z \in X,$
	
	\[d_\eta(x,z) \leq \eta(x,z)[d_\eta(x,y) + d_\eta(y,z)]. \]
	Hence $(X, d_\eta )$ is an $\eta$-cone metric space.

\end{proof}
	
	\begin{remark}
If we repeat Example \ref{example1} with $\eta(x,y)=1+\sin(x)+\sin(y)$, we also conclude that $(X, d_\eta )$ is an $\eta$-cone metric space. 
	\end{remark}

\begin{remark}
In Example \ref{example1}, note that $$ [d_\eta (1, 2) + d_\eta(2, 3)] = (80 + 600)(1,\alpha) \preceq  1000(1,\alpha)=d_\eta(1, 3) ,$$ i.e. the property (d3) holds only when we multiply by the factor $\eta(1,3)$. The property (d3) can be regarded as a "local" triangle inequality in the sense that, given, for any $x,y,z \in X,$ $d_\eta(x,y)$ and $d_\eta(x,y)+d_\eta(y,z)$, a "triangle inequality like" can always be achieved via a scalar function $\eta.
$
\end{remark}

\begin{example}
Let $X = C([a, b], \mathbb{R})$ be the space of all continuous real valued functions define on the interval $[a, b]$, $E=\mathbb{R}$ and $P=\mathbb{R}^+$. Then $(X, d_\eta )$ is an $\eta$-cone metric space with 

\[d_\eta(x,y) = \sup_{t\in [a,b]}|x(t)-y(t)|^2 \text{ and } \eta(x,y)= |x(t)| + |y(t)|+2.\]
\end{example}

The concepts of convergence, Cauchy sequence and completeness can easily be extended to the
case of an $\eta$-cone metric space.

\begin{definition}(Compare \cite{huang})
	Let $(x_n)$ be a sequence in an $\eta$-cone metric space $(X,d_\eta)$.
	\begin{itemize}
		\item[(a)]  $(x_n)$ is convergent to $x\in X$ and we denote $\underset{n\to \infty}{\lim}x_n=x$, if for every $c \in E$ with $c\gg \theta$, there exists $n_0 \in \mathbb{N}$ such that 
		$$ \forall \ n \geq n_0 \quad d_\eta(x_n,x )\ll c ;$$
		\item[(b)] $(x_n)$ is called Cauchy if for every $c \in E$ with $c\gg \theta$, there exists $n_0 \in \mathbb{N}$ such that 
		$$ \forall \ n,m \geq n_0 \quad d_\eta(x_n,x_m )\ll c ;$$
		
		\item[c)]  If every Cauchy sequence is convergent in $X$, then $X$ is called a complete
		$\eta$-cone metric space.
	\end{itemize}
\end{definition}

We state the following lemma without proof, as the proof is merely a copy of the proof of \cite[Lemma 1]{huang}:

\begin{lemma}\label{lemma1}
	Let $(X, d_\eta)$ be an $\eta$-cone metric space, $P$ be a normal cone with normal constant $K$.
	Let $\{x_n \}$ be a sequence in $X$. Then $\{x_n \}$ converges to $x$ if and only if $$d_\eta(x_n , x) \to \theta \ \ \text{ as } n \to \infty.$$
\end{lemma}

The next result, corollary of Lemma \ref{lemma1} above is a reformulation of \cite[Lemma 2]{huang}, as a uniform bound is required for the $\eta$-function to assure uniqueness for the limit of a convergent sequence in an $\eta$-cone metric space.
\begin{lemma}\label{lemma2}
	Let $(X, d_\eta)$ be an $\eta$-cone metric space such that $\eta$ is a bounded function, $P$ be a normal cone with normal constant $K$.
	Let $\{x_n \}$ be a sequence in $X$. If $\{x_n \}$ converges to $x$ and $\{x_n \}$ converges to $y$, then $x = y$.
	That is the limit of $\{x_n \}$ is unique.
\end{lemma}

\begin{proof}
	For any $c \in E$ with $0 \ll c$, there is $N$ such that for all $n > N$ , $d_\eta(x_n , x)\preceq c$ and
	$d_\eta(x_n , y) \ll c$. We have
	
	\[d_\eta(x,y) \preceq \eta(x,y)[d_\eta(x , x_n)+d_\eta(x_n , y)] \preceq 2Mc ,\]
	
	where $\eta(x,y)\leq M$ whenever $x,y\in X.$ Hence 
	$\|d(x, y)\| \leq 2MK\|c\|.$ Since $c$ is arbitrary $d_\eta(x, y) = \theta$; therefore $x = y$.
\end{proof}

We carry on with the topological properties of $\eta$-cone metric spaces with the following lemma:

\begin{lemma}\label{lemma3}
Let $(X, d_\eta)$ be an $\eta$-cone metric space, $\{x_n \}$ be a sequence in $X$. If $\{x_n\}$ converges
to $x$ and $\sup_{n,m}\eta(x_n,x_m) \leq L$ for some positive constante $L$, then $\{x_n \}$ is a Cauchy sequence.
\end{lemma}

\begin{proof}
	For any $c \in E$ with $0 \ll c$, there is $N$ such that for all $n, m > N$ , $d_\eta(x_n , x) \ll \frac{1}{2L}c$
	and $d_\eta(x_m , x) \ll \frac{1}{2L}c$. Hence
	 $$d_\eta(x_n , x_m )\preceq \eta(x_n,x_m)[d_\eta(x_n , x) + d(x_m , x)]\preceq c.$$ Therefore $\{x_n\}$ is a Cauchy sequence.
\end{proof}

A characterisation of Cauchy sequences is given by the following
\begin{lemma}\label{lemma4}
$(X, d_\eta)$ be an $\eta$-cone metric space, $P$ be a normal cone with normal constant $K$.
	Let $\{x_n\}$ be a sequence in $X$. Then $\{x_n\}$ is a Cauchy sequence if and only if $$d_\eta(x_n , x_m ) \to \theta\ \ (n, m \to \infty).$$
	
\end{lemma}

\begin{proof}
See the proof of \cite[Lemma 4]{huang}.
\end{proof}


We now state the two following lemmas, which proofs are copies of the proofs of \cite[Lemma  2.15]{Gabaq}, \cite[Lemma 2.16 ]{Gabaq}.

\begin{lemma}\label{lemma11}
	Let $(X,d_\eta)$ be an $\eta$-cone metric space. For each $c\in E$ with $ \ c\gg \theta$, there exists $\sigma>0$ such that $x\ll c$ whenever $\|x\|<\sigma, \ x\in E$.
\end{lemma}

\begin{lemma}\label{lemma22}
	Let $(X,d_\eta)$ be an $\eta$-cone metric space over a cone $P$. Then for each $c_1, c_2 \in Int(P)$, there exists $c\in Int(P)$ such that $c_1-c \in Int(P)$ and $c_2-c \in Int(P)$. 
\end{lemma}

Now, we state the following lemmas, which proof is merely a copy of the proof of
 \cite[Theorem 2.17]{Gabaq}.

\begin{proposition}
	Every  $\eta$-cone metric space is a topological space.
\end{proposition}
\begin{proof}
	For $c\gg \theta$, and $x\in (X,d_\eta)$ let $$B_\eta(x,c)=\{ y\in X: d_\eta(x,y) \ll c \}$$
	and 
	\[
	\mathcal{B} = \{ B_\eta(x,c): x\in X, c\gg \theta\}.
	\]
	Then the collection
	\[
	\mathcal{T}_\eta=\{U\subset X: \forall x\in U, \exists c\gg \theta, B_{d_\eta}(x,c)\subset U \}
	\]
	is a topology on $X$. Indeed,
	\begin{enumerate}
		\item $\emptyset$ and $X$ belong to $\mathcal{T}$,
		\item let $U,V \in \mathcal{T}$ and let $x\in U\cap V$. Then there exist $c_1\gg \theta$ and $c_2\gg \theta$ such that $B_{d_\eta}(x,c_1)\subset U$ and $B(x,c_2)\subset V$. 
		By Lemma \ref{lemma22}, there exists $c\in Int(P)$ such that $c_1-c \in Int(P)$ and $c_2-c \in Int(P)$. Then, it is clear that $x\in B_{d_\eta}(x,c)\subset U\cap V$, hence $U\cap V \in \mathcal{T}$.
		\item let $(U_\alpha)_\alpha$ be a family of sets from $\mathcal{T}$. We consider $x\in \underset{\alpha}{\cup}\ U_\alpha$. There exists $\alpha_{{}_0}$ such that $x\in U_{\alpha_{{}_0}}$. Hence, find $c\gg \theta$ such that $$x\in B_q(x,c)\subset U_{\alpha_{{}_0}}\subset \underset{\alpha}{\bigcup} \ U_\alpha,$$ that is $\underset{\alpha}{\cup}\ U_\alpha \in \mathcal{T}$.
		
\end{enumerate}
This completes the proof.
\end{proof}

\begin{definition}
	
	Let $(X,d_\eta)$ be an $\eta$-cone metric space, the collection

	\[
	\mathcal{T}_\eta=\{U\subset X: \forall x\in U, \exists c\gg \theta, B_\eta(x,c)\subset U \}
	\]
	will be called the natural topology on $(X,d_\eta)$.
\end{definition}

Before we mention our next result, we recall the present one:

\begin{lemma}(Compare \cite[Lemma 1.4]{Gabaq})\label{yes}
	Let $(X,d_\eta)$ be an $\eta$-cone metric space. Then we  have:
	
	\begin{itemize}
		\item[a)] $Int(P)+Int(P) \subset Int(P)$ and $\lambda Int(P) \subset Int(P)$ for any positive real number $\lambda$.
		
		\item[b)] For any given $c\gg \theta $ and $c_0 \gg \theta$, there exists $n_0 \in \mathbb{N}$ such that $\frac{1}{n_0}c_0\ll c$.
		
		\item[c)] If $(a_n)$ and $(b_n)$ are sequences in $E$ such that $a_n\to a$, $b_n\to b$ and $a_n\preceq b_n$ for all $n\geq 1$, then $a\preceq b$.
	\end{itemize}
\end{lemma}

\begin{lemma}
	Every $\eta$-cone metric space $(X,d_\eta)$ is first countable.
\end{lemma}

\begin{proof}
	Let $p\in X$. Fix $c\in Int(P)$. We show that $$\mathcal{B}_p = \left\lbrace B_\eta\left(p,\frac{1}{n}c\right): n\in \mathbb{N}\right\rbrace$$ is a local base at $p$. Let $U$ be an open set containing $p$. There exists $c_1 \in Int(P)$ such that $B_{d_\eta}(p,c_1)\subset U$. We know by Lemma \ref{yes} that we can find $n_0\in \mathbb{N}$ such that $\frac{c}{n_0}\ll c_1$. Hence $B_{d_\eta}(p,\frac{c}{n_0})\subset B_{d_\eta}(p,c_1) $. This completes the proof.
	
\end{proof}


\begin{definition}
	A map $T : (X_1,d_{\eta_{{}_1}}) \to (X_2,d_{\eta_{{}_2}})$ between $\eta$-cone metric spaces is called continuous at $x \in X_1$, if each $V \in \mathcal{T}_{\eta_{{}_2}}$ containing $Tx$, there exists $U \in \mathcal{T}_{\eta_{{}_1}} $ containing $x$ such that $T (U ) \subset V$. If $T$ is continuous at each $x \in X_1$ then it is called continuous.
\end{definition}

\begin{definition}
	 A map $T : (X_1,d_{\eta_{{}_1}}) \to (X_2,d_{\eta_{{}_2}})$ between $\eta$-cone metric spaces is continuous is called
	sequentially continuous if $\{x_n\} \subset X, x_n \to x$ implies $T x_n \to T x$.
\end{definition}

\begin{proposition}
	A map $T : (X_1,d_{\eta_{{}_1}}) \to (X_2,d_{\eta_{{}_2}})$ between $\eta$-cone metric spaces is continuous if and only if T is sequentially continuous.
\end{proposition}

\begin{proof}
	Assume $x_n \to x$ and let $\theta \ll c$. Since $T$ is continuous at $x \in X$, then find $\theta \ll c_1$ such that $T (B_{\eta_{{}_1}}(x, c_1 )) \subset B_{\eta_{{}_2}}(T x, c).$ By convergence of $x_n$, find $n_0$ such that $d_{\eta_{{}_1}}(x_n , x) \ll c_1$ $\forall n \geq n_0$. But then $d_{\eta_{{}2}}(T x_n , T x)\ll c, \forall n \geq n_0.$ Since $(X_1,d_{\eta_{{}_1}})$ is a first countable topological space, then the converse holds.
\end{proof}

\begin{definition}
	Let $( X , d_\eta )$ be an $\eta$-cone metric space. If for any sequence $\{x_n\}$ in $X$, there
	is a subsequence $\{x_{n_i} \}$ of $\{x_n \}$ such that $\{x_{n_i} \}$ is convergent in $X$. Then $X$ is called a
	sequentially compact $\eta$-cone metric space.
\end{definition}

\section{More on topological properties.}

In this section, we discuss some properties of closed sets in $\eta$-cone metric spaces.

\begin{lemma}\label{lemmapro}
	Let $( X , d_\eta )$ be an $\eta$-cone metric space, $\mathcal{T}_\eta$ be the natural topology on $( X , d_\eta )$ and $A$ be a subset of $X$.
	$A$ is closed\footnote{A closed set is the complement of an element of $\mathcal{T}_\eta$} if and only if for any sequence $\{ x_n \}$ in $A$ which converges to $x$, we have $x \in A$.
\end{lemma}

\begin{proof}
	Assume that $A$ is closed and let $\{ x_n \}$ be a sequence in $A$ such that $x_n \to x$. Let us prove
	that $x \in A$. Assume not, i.e. $x \notin A$. Since $A$ is closed, then there exists $c \gg \theta$ such that $B_\eta( x , c) \cap A = \emptyset$ . Since $\{ x_n \}$ converges
	to $x$, then there exists $N \geq 1$ such that for any $n \geq N$ we have $x_n \in B_\eta( x , c)$ . Hence $x_n \in B_\eta( x , c) \cap A$, which leads to a
	contradiction. Conversely assume that for any sequence $\{ x_n \}$ in $A$ which converges to $x$, we have $x \in A$. Let us prove that $A$
	is closed. Let $x \notin A$. We need to prove that there exists $c\gg \theta$ such that $B_\eta( x , c) \cap A = \emptyset$. Assume not, i.e. for any $c\gg \theta$, we
	have $B_\eta( x , c) \cap A \neq \emptyset$. So for any $n \geq 1$, choose $x_n \in B_\eta\left( x , \frac{1}{n}c\right) \cap A$. Clearly we have $\{ x_n \}$ converges to $x$. Our assumption on
	$A$ implies $x \in A$, a contradiction. So $A$ is closed.
\end{proof}

\begin{proposition}
	Let $( X , d_\eta )$ be an $\eta$-cone metric type space and $\mathcal{T}_\eta$ be the topology defined above. Moreover, suppose that $\sup_{x,y\in X}\eta(x,y) = L <\infty.$
	Then for any nonempty subset $A \subset X$, if we define $\bar{A}$ to be the intersection of all closed subsets of $X$ which contains $A$, then for any $x \in \bar{A}$ and for any $c\gg \theta$, we have
	\[B_\eta( x , c) \cap A \neq \emptyset .\]

\end{proposition}

\begin{proof}
	Clearly $\bar{A} $ is the smallest closed subset which contains $A$. Set
	\[  A^* = \{ x \in X ; \text{ for any } c\gg\theta , \text{ there exists } a \in A \text{ such that }: d_\eta ( x , a ) \ll c \}.\] 
	
We have $A \subset A^*$. Next we prove that $A^*$ is closed. For this we use Lemma \ref{lemmapro}. Let $\{ x_n \}$ be a sequence in $A^*$ such that $\{ x_n \}$ converges to $x$. Let us prove that $x \in A^*$. Let $c\gg \theta$. Since $\{ x_n \}$ converges to $x$, there exists $N \geq 1$ such that for any $n \geq N$ we have $d_\eta(x,x_n) \ll \frac{1}{2(L+1)}c$. Since $x_n \in A^*$ , there exists $a \in A$ such that $d_\eta(a,x_n) \ll \frac{1}{2(L+1)}c$.	Hence
\[d_\eta(x,a) \preceq \eta(x,a) [d_\eta(x,x_n)+d_\eta(x_n,a)] \preceq K \left[ \frac{1}{2(L+1)}c + \frac{1}{2(L+1)}c\right]  \preceq c\]
which implies $x\in A^*$. Therefore $A^*$  is closed and contains $A$. By definition $\bar{A}\subset A^*$, which concludes the proof.

\end{proof}

\section{Fixed point theory.}

In this section we shall prove some fixed point theorems of contractive mappings. Our first theorem is an analogue of Banach contraction principle in the setting of $\eta$-cone metric
space. Throughout this section, for the mapping $T : (X,d_\eta) \to (X,d_\eta)$ and $x_0 \in X$, the set $I(x_0,T)=\{x_0,Tx_0,T^2x_0, \cdots \}$ represents the orbit of $x_0.$

Before we proceed to our first fixed point result, there is an important observation that ought to be made. As we know, the metric $d$ on a metric space $(X,d)$ possesses the continuity property, which is equivalent to the sequential continuity since any metric space is first countable. However, as the next example will demonstrate, $\eta$-cone metrics are not always continuous.

\begin{example}(Compare \cite[Example 3]{alex})
	Let $X = \mathbb{N} \cup \{\infty\}$ and let $D : X \times X \to \mathbb{R}$ be defined by:
$$
D(m,n)=
\begin{cases}
0,\ \ \ \ \ \ \ \ \ \ \ 
\text{if } m=n,\\
\left| \frac{1}{m}-\frac{1}{n} \right|,  \ \text{if } m \text{ and } n \text{ are even or } mn = \infty\\
5\ \  \ \ \ \ \ \ \ \ \ \ 
\text{if } m \text{ and } n \text{ are odd and } m\neq n\\
2\ \ \ \ \ \ \ \ \ \ \ 
\text{ otherwise.}
\end{cases}
$$
Thus, $(X, d_\eta)$ is an $\eta$-cone metric space where\footnote{Note that here $E=\mathbb{R}, P=[0,\infty)$.} $\eta(x,y)=3$ and $d_\eta(x,y)=D(x,y)$ whenever $x,y \in X.$  Considering the sequence $\{x_n\}$ defined by $x_n=2n$ and setting $\frac{1}{\infty}:=0$, we have that
$$ d_\eta(x_n,\infty) = \frac{1}{2n} \to 0 \quad (\text{ i.e. } x_n \to \infty).$$
However
$$ d_\eta(x_n,1) = 2, \text{ hence } \lim\limits_{n \to \infty}d_\eta(x_n,1)= 2 \neq  1= d_\eta(\infty,1). $$
\end{example}

\begin{theorem}\label{theorem1}
	Let $ (X,d_\eta)$ be a complete $\eta$-cone metric space such that $d_\eta$
	is continuous. Suppose the mapping $T : (X,d_\eta) \to (X,d_\eta)$ satisfies the contractive condition
	\begin{equation}\label{condition1}
		d_\eta(Tx,Ty) \preceq k d_\eta(x,y),
	\end{equation}
	 
for some $k \in [0,1)$	and for all $x,y \in X.$ 

Moreover, for any $x_0 \in X$, suppose that $\lim\limits_{n,m\to \infty}\eta(x_n,x_m)<\frac{1}{k}$ where $x_n,x_m \in I(x_0,T) $. Then $T$ has exactly one fixed point $x^*.$ Moreover for each $x \in X$, $T^n x \to x^*$.
\end{theorem}

\begin{proof}
	We choose any $x_0 \in X$ be arbitrary, define the sequence $\{x_n\}$ by $x_n = T^n x_0$. Then by successively applying inequality \eqref{condition1}, we obtain:

\begin{equation}\label{condition2}
	d_\eta(x_n,x_{n+1}) \preceq k^n d_\eta(x_0,x_1).
\end{equation}
	So for $m > n$, using property (d3) and condition \eqref{condition2}
	\begin{align}\label{formalise}
		d_\eta (x_n, x_m)  \preceq & \ \eta(x_n, x_m)k^
		n
		d_\eta(x_0, x_1) + \eta(x_n, x_m)\eta(x_{n+1}, x_m)k^
		{n+1}
		d_\eta(x_0, x_1) +\cdots + \nonumber \\
		 & \eta(x_n, x_m)\eta(x_{n+1}, x_m)\eta(x_{n+2}, x_m)\cdots \eta(x_{m-2}, x_m)\eta(x_{m-1}, x_m)k^{
		 m-1}
		 d_\eta (x_0, x_1) \nonumber \\
		 \preceq & d_\eta(x_0, x_1)[
		 \eta(x_1, x_m)\eta(x_2, x_m)\cdots \eta(x{n-1}, x_m)\eta(x_n, x_m)k^n + \nonumber \\
		 & \eta(x_1, x_m)\eta(x_2, x_m)\cdots \eta(x_n, x_m)\eta(x_{n+1}, x_m)k^{
		 n+1} + \cdots + \nonumber \\
		 & \eta(x_1, x_m)\eta(x_2, x_m)\cdots \eta(x_n, x_m)\eta(x_{n+1}, x_m)\cdots \eta(x_{m-2}, x_m)θ(x_{m-1}, x_m)k^{m-1}].
		 		 	\end{align}
Since $\lim\limits_{n,n\to \infty} \eta(x_{n+1}, x_m)k < 1$, hence the series	$\sum_{n=0}^{\infty} k^n \prod_{i=1}^{n}\eta(x_i,x_m)$ converges by ratio test for each $m\geq 1 $. Let:
	
	\[  S= \sum_{n=0}^{\infty} k^n \prod_{i=1}^{n}\eta(x_i,x_m), \quad S_n = \sum_{j=1}^{n} k^j \prod_{i=1}^{j}\eta(x_i,x_m).   \]
Thus for $m > n$, the above inequality \eqref{formalise} becomes:	

\[ d_\eta(x_n, x_m) \preceq [S_{m-1} - S_n]d_\eta(x_0, x_1).
  \]
We get $$\|d_\eta(x_n , x_m )\| \leq K | S_{m-1} - S_n|  \|d_\eta(x_1 , x_0 )\|.$$	
This implies $d_\eta(x_n , x_m ) \to \theta \ (n, m \to \infty).$ Hence
$\{x_n \}$ is a Cauchy sequence. By the completeness of $X$, there is $x^* \in X$ such that $x_n \to x^*$.

\begin{align*}
	d_\eta(Tx^*, x^*) \preceq & \eta(Tx^*, x^*)[d_\eta(Tx^*, x_n) + d_\eta(x_n, x^*)] \\
	\preceq &  \eta(Tx^*, x^*)[kd_\eta(x^*, x_{n-1}) + d_\eta (x_n, x^*)].
\end{align*}
Then 
$$ \|d_\eta(T x^* , x^*)\|\leq \ K \ |\eta(Tx^*, x^*)|\ [ k \| d_\eta(x^*,x_{n-1})\|  + \|d_\eta(x_{n} , x^*)\| ]\to 0.$$

Hence $\|d(T x^* , x^* )\| =0.$ This implies $T x^* =x^*$ . So $x^*$ is a fixed point of $T$.
Now if $y^*$ is a fixed point of $T$ , then

$$d(x^*, y^* ) = d(T x^* , T y^* ) \preceq k d(x^*, y^* ).$$
Therefore $\|d(x^*, y^* )\| = 0$ and $x^*=y^*.$
\end{proof}


The proof of the following corollary is immediate.
\begin{corollary}
	Let $ (X,d_\eta)$ be a complete $\eta$-cone metric space such that $d_\eta$
	is continuous. Suppose the mapping $T : (X,d_\eta) \to (X,d_\eta)$ satisfies, for some positive integer $n$, the contractive condition
	\begin{equation}\label{condition2r}
	d_\eta(T^nx,T^ny) \preceq k d_\eta(x,y),
	\end{equation}
	
	for some $k \in [0,1)$	and for all $x,y \in X.$ 
	
	Moreover, for any $x_0 \in X$, suppose that $\lim\limits_{n,m\to \infty}\eta(x_n,x_m)<\frac{1}{k}$ where $x_n,x_m \in I(x_0,T) $. Then $T$ has exactly one fixed point $x^*.$ 
\end{corollary}

\begin{example}
Let $X =[0, +\infty), E = \mathbb{R}$ ,
	and $P= [0, +\infty)$. Let us define, for all $x,y \in X$ 
$d_\eta (x, y) : X \times X \to \mathbb{R}$ and $\eta : X \times X \to [1, \infty)$ as:
	\[d_\eta(x,y) = (x-y)^2, \text{ and } \eta(x,y)= x+y+2.   \]  
	Then $d_\eta$
	is a complete $\eta$-cone metric on $X$. Define $T : X \to X$ by $Tx = \frac{x}{2}$. We have:
	\[d_\eta(Tx,Ty) = d_\eta\left( \frac{x}{2}, \frac{y}{2}\right)= \left( \frac{x}{2}-\frac{y}{2}    \right)^2 \leq \frac{1}{3}(x-y)^2 = kd_\eta(x,y).
	\]
	
Note that for each $x \in X$, $T^nx = \frac{x}{2^n}$. Thus we obtain:	
\[   \lim\limits_{n,m\to \infty}\eta(x_n,x_m) = \lim\limits_{n,m\to \infty} \left( \frac{x}{2^n}+\frac{x}{2^m} + 2 \right)  <  3.         \]
Therefore, all conditions of Theorem \ref{theorem1} are satisfied hence $T$ has a unique fixed point.
\end{example}

\begin{theorem}\label{strict}
Let $ (X,d_\eta)$ be a sequentially compact $\eta$-cone metric space such that $d_\eta$
	is continuous, $\sup_{x,y}\eta(x,y)<\infty$ and the underlying cone $P$ is regular. Suppose the map $T : (X,d_\eta) \to (X,d_\eta)$ satisfies the contractive condition
	\begin{equation}\label{conditionstrict}
	d_\eta(Tx,Ty) \prec d_\eta(x,y),
	\end{equation}
	for all $x,y \in X, x\neq y.$ Then $T$ has a unique fixed point in $X$.
	
\end{theorem}

\begin{proof}
	Let $x_0\in X$ be arbitrary and construct the sequence $\{x_n\}$ such that $x_{n+1}=Tx_n.$ Moreover, we may assume, without loss of generality that $x_n\neq x_{m}$ for $n\neq m$. By setting $d_n = d_\eta(x_n , x_{n+1} )$, then, using condition \eqref{conditionstrict}, we write
	\[d_{n+1} = d_\eta(x_{n+1} , x_{n+2} ) = d_\eta(Tx_n , Tx_{n+1} ) \prec d_\eta(x_n , x_{n+1} ) =d_n.  \]
	
	Therefore $d_n$ is a decreasing sequence bounded below by $\theta$. Since $P$ is regular, there is $d^*\in E$ such that $d_n \to d^* (n  \to \infty)$. From the sequence compactness of $X$, there is a
	subsequence $\{x_{n_i} \}$ of $\{x_n \}$ and $x^*\in X$ such that $x_{n_i} \to x^* (i \to \infty)$. We have

	\[ d_\eta(Tx_{n_i},Tx^*) \prec d_\eta(x_{n_i},x^*) \quad i=1,2,\cdots.\]
So
\[ \|d_\eta(Tx_{n_i},Tx^*)\| \leq K \|d_\eta(x_{n_i},x^*)\| \ \to 0 \ (i \to \infty) \]	
	
where $K$ is the normal constant of $E$. Hence $Tx_{n_i} \to Tx^* (i \to \infty)$. In a similar manner, on establishes that  $T^2x_{n_i} \to T^2x^* (i \to \infty)$. Using the continuity of $d_\eta$, we write
\[\footnote{This is guaranteed by the boundedness of $\eta$.}d_{n_i}=d_\eta(Tx_{n_i},x_{n_i} )\to d_\eta(Tx^*,x^*)=d^* \text{ and } d_\eta(T^2x_{n_i},Tx_{n_i} )\to d_\eta(T^2x^*,Tx^*) \ (i \to \infty). \]	

Moreover, if we assume that $d^*\neq \theta,$ we have

\[ d^*= \lim\limits_{i\to \infty }d_{n_i+1}= \lim\limits_{i\to \infty }d_\eta(T^2x_{n_i},Tx_{n_i})  = d_\eta(T^2x^*,Tx^*) \prec d_\eta(Tx^*,x^*)=d^* , \]
--a contradiction, so $d^* =\theta, $ i.e. $Tx^*=x^*$ and $x^*$ is a fixed point of $T$. The uniqueness of the fixed point naturally comes from condition \eqref{conditionstrict}.

\end{proof}


As we mentioned earlier $\eta$-cone metric spaces have a metric like structure. Indeed we
have the following result.

\begin{theorem}(Compare \cite[Theorem 2.6]{khamsi})
	Let $(X,d_\eta)$ be an $\eta$-cone metric space (over the Banach space $E$ with the $K$-normal cone $P$). The mapping $D_\eta:X\times X \to [0,\infty)$ defined by $D_\eta(x,y)=\|d_\eta(x,y)\|$ satisfies the following properties
	\begin{itemize}
		\item[(D1)] $D_\eta(x,x)=0$ for any $x\in X$;
		\item[(D2)] $D_\eta(x,y)=D_\eta(y,x)$ for any $x,y\in X$;
		\item[(D3)] $D_\eta(x,y)\leq K\eta(x,y)  \big(  D_\eta(x,z)+D_\eta(z,y)\big)$, for any points $x,y,z\in X$.
	\end{itemize}
\end{theorem}
Note that property (D3) does not give the classical triangle inequality satisfied by a distance and there are many examples where the triangle inequality fails. We are therefore led to the following definition.

\begin{proof}
The proofs of (D1) and (D2) are easy and therefore left to the reader. In order to prove (D3), let $x,y,z\in X,$ and since $(X,d_\eta)$ is an $\eta$cone metric space, we have:

\[ d_\eta(x,y) \preceq \eta(x,y) [d_\eta(x,z)+d_\eta(z,y)] .\]

Since $P$ is normal with constant $K$ we get

\begin{align*}
	\|d_\eta(x,y)\| & \leq K |\eta(x,y)| [\|d_\eta(x,z)+d_\eta(z,y)\|] \\
	&  \leq  K |\eta(x,y)| [\|d_\eta(x,z)\| +\|d_\eta(z,y)\|] 
\end{align*}
 i.e.
 
 \[D_\eta(x,y)\leq K\eta(x,y)  \big(  D_\eta(x,z)+D_\eta(z,y)\big)  .\]
 This completes the proof.

\end{proof}

\begin{definition}
	Let $X$ be a nonempty set, let the function $D:X\times X \to [0,\infty)$ and a function $\eta : X \times X \to [1, \infty)$ satisfy the following properties:
	\begin{itemize}
		\item[(D1)] $D(x,x)=0$ for any $x \in X$;
		\item[(D2)] $D(x,y)=D(y,x)$ for any $x,y\in X$;
		\item[(D3)] $D(x,y) \leq \eta(x,y) \big( D(x,z)+D(z,y) \big)$ for any points $x,y,z\in X$.
	\end{itemize}
	The triplet $(X,D,\eta)$ is called an \textbf{$\eta$-metric space}.
\end{definition}

It is obvious that $\eta$-metric spaces are natural extensions of metric type spaces.

\begin{example}
	Let $X=\{ 1,2,3 \}$ and the mapping $D:X\times X \to [0,\infty)$ defined by $D(1,2)=1/5,\ D(2,3)=1/4,\ D(1,3)=1/2$, $D(x,x)=0$ for any $x \in X$ and $D(x,y)=D(y,x)$ for any $x,y\in X$.
	Since 
	\[
	\frac{1}{2} = D(1,3)> D(1,2)+D(2,3)=\frac{9}{20},
	\]
\end{example}
then we conclude that $X$ is not a metric space. Nevertheless, with $\eta(x,y)=\frac{x}{2}+y$, it is very easy to check that $(X,D,\eta)$ is an $\eta$-metric space.

The next corollary follows immediately form the definition of a metric type space.

\begin{corollary}
Let $(X,D,\eta)$ be an $\eta$-metric space. If $\sup_{x,y}\eta(x,y)<\infty$, then $(X,D,\eta)$ is a metric type space.
\end{corollary}

The concepts of convergence, Cauchy sequence and completeness can easily be extended to the
case of an $\eta$(cone)-metric space.


We now try and formalise the inequality \eqref{formalise} (which is a generalization of property (D3)) in the case of $\eta$-metric space, namely we have the following lemma, whose proof is straightforward and shall therefore be omitted:

\begin{lemma}\label{triangle}
Let $(X,D,\eta)$ be an $\eta$-metric space. Let $x,y,z_i, i=1,2,\cdots,n \in X, \ n\geq 2$. Then we have the recursion:
\[ D(x,y)\leq \eta(x,y)\left( D(x,z_1)+\sum_{j=1}^{n-1}\left(\prod_{i=1}^{j}\eta(z_i,y)D(z_j,z_{j+1}) \right) + \prod_{i=1}^{n-1} \eta(z_i,y)D(z_n,y) \right). \]

\end{lemma}

\begin{corollary}\label{cochi1}
	Let $(X,d_\eta)$ be an $\eta$-cone metric space such that $\sup_{x,y}\eta(x,y)<\infty$. Consider the $\eta$-metric space $(X,D,\eta)$ where $D(x,y)=\|d_\eta(x,y)\|$. Let $(y_n)$ be a sequence in $(X,D,\eta)$ such that 
\begin{equation}\label{cochi}
D(y_n,y_{n+1}) \leq \lambda D(y_{n-1},y_n)
\end{equation}
for some $0<\lambda<1$. If $\lim\limits_{n,m\to \infty}\eta(y_n,y_m)<\frac{1}{\lambda}$,
then $(y_n)$ is Cauchy.
\end{corollary}

\begin{proof}
	By \eqref{cochi}, we have
	\begin{equation}\label{eq}
	 D(y_n,y_{n+1}) \leq \lambda^n D(y_0,y_1)
	 \end{equation}
	
	Let $m>n \in \mathbb{N}$. Using Lemma \ref{triangle}, we write

	\[ D(y_n,y_m)\leq \eta(y_n,y_m)\left( D(y_n,y_{n+1})+\sum_{j=1}^{m-n-2}\left(\prod_{i=1}^{j}\eta(y_i,y_m)D(y_j,y_{j+1}) \right) + \prod_{i=1}^{m-n-2} \eta(z_i,y)D(y_{m-1},y_m) \right). \]

and by the inequality \eqref{eq}, we obtain

\[ D(y_n,y_m)\leq \eta(y_n,y_m)\left( \lambda^n D(y_0,y_1)+\sum_{j=1}^{m-n-2}\left(\prod_{i=1}^{j}\eta(y_i,y_m)\lambda^j D(y_0,y_1) \right) + \prod_{i=1}^{m-n-2} \eta(z_i,y)\lambda^{m-1} D(y_0,y_1) \right). \]

Since $\lim\limits_{n,n\to \infty} \eta(y_{n+1}, y_m)\lambda < 1$ so that the series	$\sum_{n=0}^{\infty} \lambda^n \prod_{i=1}^{n}\eta(x_i,x_m)$ converges by ratio test for each $m\geq 1 $. Let:

\[  S= \sum_{n=0}^{\infty} \lambda^n \prod_{i=1}^{n}\eta(x_i,x_m), \quad S_n = \sum_{j=1}^{n} \lambda^j \prod_{i=1}^{j}\eta(y_i,y_m).   \]
Thus for $m > n$, the above inequality becomes:	

\[ D(y_n, y_m) \leq [S_{m-1} - S_n]D(y_0, y_1).
\]

This implies $d_\eta(y_n , y_m ) \to \theta \ (n, m \to \infty)$\footnote{In fact $ D(y_n , y_m ) \to 0 \ (n, m \to \infty).$}. Hence
$\{y_n \}$ is a Cauchy sequence.

\end{proof}

We conclude this manuscript with the following interesting fixed point result. Let $(X,d_\eta)$ be an $\eta$-cone metric space such that $d_\eta$ is continuous. For $x, y \in X$ we set $ D_\eta(x,y )=\|d_\eta(x,y)\|$. Then $(X,D_\eta)$ is the $\eta$-metric space generated by $(X,d_\eta)$. We have:

\begin{theorem}\label{fin}
	Let $(X, D_\eta)$ be the complete $\eta$-metric space generated by the $\eta$-cone metric space $(X,d_\eta)$ and $T : X \to X$ a self mapping on $X$ such that for each $x, y \in X$:
	
	\begin{align}\label{conditionfin}
		D_\eta(T x, T y) & \leq \alpha (x, y) D_\eta(x, y) + \beta (x, y) D_\eta(x, T x) + \gamma (x, y) D_\eta(y, T y) \nonumber\\
		& +\delta (x, y) [D_\eta(x, T y) + D_\eta(y, T x)] ,
	\end{align}
where $\alpha,\beta, \gamma, \delta$ are functions from $X \times X$ into $[0, 1)$ such that

\begin{equation}
	\lambda = \sup \{\alpha (x, y) + \beta (x, y) + \gamma (x, y) + 2\eta (x, y)\delta(x,y) : x, y \in X\} < 1.
\end{equation}

If $\sup_{x,y}\eta(x,y)<\frac{1}{\lambda}$,
then $T$ has a unique fixed point in $X.$
\end{theorem}

\begin{proof}
Fix $x_0 \in X$ and construct the sequence $\{x_n\}$ such that $x_{n+1}=Tx_n.$
 From \eqref{conditionfin},
 \begin{align*}
 	D_\eta(x_n,x_{n+1}) =	D_\eta(Tx_{n-1},Tx_{n}) & \leq \alpha D_\eta(x_{n-1},x_n)+\beta D_\eta(x_{n-1},x_n) + \gamma D_\eta(x_n,x_{n+1}) \\ & +  \delta[D_\eta(x_{n-1},x_{n+1})+D_\eta(x_n,x_n) ]
 \end{align*}

where $\alpha.\beta, \gamma$ and $\delta$ are evaluated at $(x_{n-1}, x_n)$. By property (D3) we have

$$D_\eta (x_{n-1} , x_{n+1} ) \leq \eta(x_{n-1} , x_{n+1}) [D_\eta (x_{n-1} , x_{n}) + D_\eta (x_n , x_{n+1} )] ,$$
which implies that 

$$ D_\eta (x_{n-1} , x_{n+1} ) \leq 2 \eta(x_{n-1} , x_{n+1}) \max\left\lbrace D_\eta (x_{n-1} , x_{n}) , D_\eta (x_n , x_{n+1} ) \right\rbrace.$$

Hence

\begin{align*}
D_\eta(x_n,x_{n+1}) & \leq (\alpha+\beta+\gamma)\max\left\lbrace D_\eta (x_{n-1} , x_{n}) , D_\eta (x_n , x_{n+1} ) \right\rbrace \\
  & + 2 \eta(x_{n-1} , x_{n+1})\delta \max\left\lbrace D_\eta (x_{n-1} , x_{n}) , D_\eta (x_n , x_{n+1} ) \right\rbrace.
\end{align*}
Then

\[D_\eta(x_n,x_{n+1}) \leq \lambda  \max\left\lbrace D_\eta (x_{n-1} , x_{n}) , D_\eta (x_n , x_{n+1} ) \right\rbrace. \]

Since $\lambda < 1$, then
\[ D_\eta(x_n,x_{n+1}) \leq \lambda  D_\eta (x_{n-1} , x_{n})  .\]

From Corollary \ref{cochi1}, we know that $\{x_n\}$ is Cauchy. Since
$(X, D_\eta)$ is complete, then there exists $x^* \in X$ such that
\[\lim\limits_{n\to \infty}x_n=x^*.\]
On the other side, using \eqref{conditionfin}, one readily sees that
\[D_\eta (T x^*, T x_n )\leq \lambda \max \left\lbrace 
D_\eta(x^*, x_n ) , D_\eta(x^*, T x^*) , D_\eta(x_n , x_ {n+1} ) , D_\eta(x^*, x_{n+1} ) , D_\eta(x_n , T x^*)
 \right\rbrace.\]	
Now, take the limit as $n \to \infty$, then by and since $d_\eta$ is continuous, we obtain
$$D_\eta(T x^*, x^*) \leq \lambda D_\eta(x^*, T x^*).
$$	
Since $\lambda<1$, then $T x^* = x^*.$

For uniqueness, assume $x, y \in X$ and $x \neq y$ are two fixed points of $T$. Using \eqref{conditionfin}, we have

\[= D_\eta(T x, T y)
\leq (\alpha + 2\delta\eta) D_\eta(x, y) \leq \lambda D_\eta (x, y) .  \]

Since $\lambda < 1$, then $ D_\eta(x, y)  = 0$, which implies $x = y.$

The proof is complete.
\end{proof}

\begin{example}
	Let $X = \left[ 0, \frac{1}{4} \right]$. Let $E=\mathbb{R}$ and $P=\mathbb{R}^+$. Then $(X, d_\eta )$ is an $\eta$-cone metric space with 
	 $$d_\eta(x,y) = (x-y)^2, \text{ and } \eta(x,y)= x+y+2. $$
	 
	 The $\eta$-metric space $(X,D_\eta)$, generated by $(X,d_\eta)$ has the same structure since $$D_\eta(x,y)=|d_\eta(x,y)|=(x-y)^2.$$
	 
	 Then $d_\eta$ is a complete $\eta$-metric on $X$. Define $T : X \to X$ by $Tx = x^2$. 
	 
	 Let $ \beta=\gamma=\delta=0$ and
	 $\alpha(x,y) =\frac{1}{4},$ we have
	 \[  D_\eta( Tx, Ty ) \leq \frac{1}{4}
	 D_\eta( x, y ). \]
	 
	 Note that for each $x \in X, T^n x = x^{2n}$. Thus we obtain:
	 \[\lim\limits_{n,m\to \infty} \eta(T^mx, T^nx)<4.\]
	 
Therefore, all conditions of Theorem \ref{fin} are satisfied hence T has a unique fixed point.	 
\end{example}

\bibliographystyle{amsplain}

\end{document}